\documentclass{template}

\usepackage{hyperref}
\usepackage[T1]{fontenc}
\usepackage[utf8]{inputenc}
\usepackage[english]{babel}
\usepackage{amsmath,amsthm,amssymb}
\usepackage{fullpage}
\usepackage{float}    
\usepackage{xcolor}	  
\usepackage{graphicx} 
\usepackage{subfig}   

\newcommand{\R}{\mathbb{R}}
\newcommand{\C}{\mathbb{C}}

\renewcommand{\AA}{\boldsymbol{A}}
\newcommand{\BB}{\boldsymbol{B}}
\newcommand{\EE}{\boldsymbol{E}}
\newcommand{\JJ}{\boldsymbol{J}}
\newcommand{\UU}{\boldsymbol{U}}
\newcommand{\Afrak}{\mathfrak{A}}
\newcommand{\Bfrak}{\mathfrak{B}}
\newcommand{\Cfrak}{\mathfrak{C}}
\newcommand{\Dfrak}{\mathfrak{D}}
\newcommand{\jj}{\boldsymbol{j}}

\newcommand{\meshsize}{\Delta x}

\newcommand{\timestepsize}{\Delta t}
\newcommand{\eps}{\varepsilon}
\newcommand{\ssigma}{\boldsymbol{\sigma}}
\newcommand{\dx}{\mathrm{d}\mathbf{x}}
\newcommand{\abs}[1]{\lvert #1 \rvert}

\newcommand{\norm}[2][]{\lVert #2 \rVert_{#1}}
\newtheorem{theorem}{Theorem}[section]
\newtheorem{lemma}[theorem]{Lemma}
\newtheorem{remark}[theorem]{Remark}
\newtheorem{algorithm}[theorem]{Algorithm}

\begin{document}

\begin{frontmatter}

\title{A time splitting method for the three-dimensional linear Pauli equation}

\author[oxaddress]{Timon S. Gutleb}\corref{mycorrespondingauthor}
\cortext[mycorrespondingauthor]{Corresponding author}
\ead{timon.gutleb@maths.ox.ac.uk}

\author[viennaaddress]{Norbert J.\ Mauser}
\ead{norbert.mauser@univie.ac.at}

\author[glasgowaddress]{Michele Ruggeri}
\ead{michele.ruggeri@strath.ac.uk}

\author[viennaaddress]{Hans Peter Stimming}
\ead{hans.peter.stimming@univie.ac.at}

\address[viennaaddress]{Research platform MMM `Mathematics - Magnetism - Materials' c/o Fak. Mathematik, Universität Wien, A-1090 Vienna
}
\address[oxaddress]{Mathematical Institute, University of Oxford, Oxford OX2 6GG, UK}
\address[glasgowaddress]{Department of Mathematics and Statistics, University of Strathclyde, Glasgow G1 1XH, UK}

\begin{abstract}
We analyze a numerical method to solve the time-dependent linear Pauli equation in three space dimensions. The Pauli equation is a semi-relativistic generalization of the Schrödinger equation for 2-spinors which accounts both for magnetic fields and for spin, with the latter missing in preceding numerical work on the linear magnetic Schr\"odinger equation. We use a four operator splitting in time, prove stability and convergence of the method and derive error estimates as well as meshing strategies for the case of given time-independent electromagnetic potentials, thus providing a generalization of previous results for the magnetic Schrödinger equation.
\end{abstract}


\end{frontmatter}


\section{Introduction}

Relativistic quantum mechanics is appropriate for the dynamics of "fast" charged particles (e.g. electrons moving close to speed of light $c$).
In the \emph{fully relativistic regime} the Dirac equation with electromagnetic potentials is the appropriate model, where the unknown is a 4-spinor including both spin and antimatter in a quantum field theory approach \cite{griffiths_introduction_2011,schwartz_quantum_2014}. In the \emph{fully non-relativistic} ("Newtonian") $c \to \infty$ regime one uses the standard Schr\"odinger equation with electric potential for the scalar wave function. In the intermediate \emph{semi-relativistic} ("Post-Newtonian") regime of a first order theory, i.e. keeping the corrections at $O(\frac{1}{c})$, the appropriate model is the Pauli equation for the 2-spinor. It is the simplest available theory that retains relativistic effects of both electromagnetism \emph{and} spin, in contrast to the scalar magnetic Schr\"odinger equation where spin is completely absent in the model. This hierarchy of approximations of the Dirac equation is laid out, e.g.
in ~\cite{schwartz_quantum_2014,mauser_semi-relativistic_2000,masmoudi_selfconsistent_2001,MauMoe2023} and in \cite{mauser_semi-relativistic_2000,masmoudi_selfconsistent_2001,MauMoe2023} specifically also for the self-consistent case of coupling to the Maxwell equations and their magnetostatic approximations. In \cite{MauMoe2023} the related Pauli-Poisson model is discussed in which the magnetic field is linear and the electric field is self-consistent. This model can formally be justified in a weak coupling limit from the linear $N$-body Pauli equation our numerical scheme applies and rigorous proofs of convergence are subject to follow-up work.\\
The Pauli equation contains a magnetic Schr\"odinger operator and a so-called Stern-Gerlach term that couples the magnetic field to the spin operators; the time dependent version reads:

\begin{equation} \label{eq:pauli}
i \hbar \, \partial_t u =
\left[
\frac{1}{2m}\left(-i\hbar \nabla-\frac{q}{c}\AA\right)^2 
+ q \, \phi
-\frac{\hbar q}{2mc} \, \ssigma \cdot \BB
\right] u.
\end{equation}
Here,
$u \in \C^2$ is a 2-spinor $(u_1,u_2)^T$ representing quantum mechanical spin up and spin down states,
$\AA \in \R^3$ and $\phi \in \R$ denote the magnetic vector potential and the electric scalar potential,
respectively, which are related to the electromagnetic fields by
\begin{equation*}
\EE = - \nabla \phi - \partial_t \AA
\quad
\text{and}
\quad
\BB = \nabla \times \AA.
\end{equation*}
Moreover, $i \in \C$ denotes the imaginary unit, i.e. $i^2=-1$,
$\ssigma = (\sigma_1, \sigma_2, \sigma_3)$ is a vector collecting the 3 Pauli matrices and the product $\ssigma\cdot\BB$ is a shorthand notation for the matrix
\begin{equation*}
\ssigma\cdot\BB = \sum_{j=1}^3 B_j \sigma_j = \begin{pmatrix} B_3 & B_1- i B_2 \\ B_1 + i B_2 & -B_3 \end{pmatrix} \in \C^{2 \times 2}.
\end{equation*}
Finally, $m$ and $q$ are the associated mass and charge,
while the positive constants $\hbar$ and $c$ are the scaled Planck constant and the speed of light respectively.
The above rendition of the Pauli equation retains all of the gauge freedom of electrodynamics and is semi-relativistic
in the sense that it is suitable for medium high velocities relative to the speed of light, cf. \cite{mauser_semi-relativistic_2000,masmoudi_selfconsistent_2001}. From the complex valued 2-spinor solution $u$ of~\eqref{eq:pauli} the physical quantities of interest are computed as quadratic quantities, e.g. the position density $n= \abs{u}^2 = u \cdot \overline{u}$
and the current density\footnote{In~\eqref{eq:currentDensity},
$\overline{u} \cdot \nabla u = \overline{u}_1 \nabla u_1 + \overline{u}_2 \nabla u_2\in \R^3$
and $u \nabla \overline{u} = u_1 \nabla \overline{u}_1 + u_2 \nabla \overline{u}_2\in \R^3$,
while $\overline{u} \ssigma u \in \R^3$ denotes the vector with components
$(\overline{u} \ssigma u)_j = \overline{u} \sigma_j u$ for all $j=1,2,3$.} which contains (divergence free) extra terms to the standard definition for the Schr\"odinger equations (cf. \cite{nowakowski1999}),
\begin{equation} \label{eq:currentDensity}
\JJ = - \frac{i \hbar}{2m} \big( \overline{u}  \cdot \nabla u - u  \cdot \nabla\overline{u} \big)
- \frac{q}{mc} \abs{u}^2 \AA
- \frac{\abs{q} \hbar}{2m} \, \nabla \times (\overline{u} \ssigma u).
\end{equation}
Lastly, we note the continuity equation connecting $n$ and $\JJ$ as well as conservation of total mass and energy:
\begin{equation*}
\partial_t n + \nabla \cdot \JJ = 0,
\end{equation*}
\begin{equation*}
m_{\mathrm{tot}}
= m \int_{\R^3} n \, \dx
= m \int_{\R^3} \abs{u}^2 \, \dx,
\end{equation*}
\begin{equation*}
\mathcal{E}
= \frac{1}{2m} \int_{\R^3} \big\lvert (-i \hbar \nabla - (q/c) \AA)u \big\rvert^2 \, \dx
+ q \int_{\R^3} \phi \abs{u}^2 \, \dx
-\frac{\hbar q}{2mc} \int_{\R^3} (\ssigma \cdot \BB) u \cdot \overline{u} \, \dx.
\end{equation*}
In this paper, we propose and analyze an exponential splitting method~\cite{splitting_acta}
for the Pauli equation~\eqref{eq:pauli}.
The scheme is an extension of analogous approaches developed for
the scalar magnetic Schr\"odinger equation~\cite{jin_semi-lagrangian_2013,caliari_splitting_2017,mzz2017}.
Our method consists of a four-term operator splitting,
where the three operator contributions appearing in the magnetic Schr\"odinger equation
(kinetic, potential, advective)
are supplemented with a fourth term accounting for spin.
The presence of this additional contribution determines a bidirectional coupling of the two equations
for the two components of $u$.

The remainder of this paper is organized as follows:
The proposed method is described in Section~\ref{sec:method};
In Section~\ref{sec:analysis} we study stability (Theorem \ref{stabilitytheorem}) and convergence (Theorem \ref{convergencetheorem}) of the method: The applicability of the scheme is demonstrated in Section~\ref{sec:numerics} with numerical experiments.

\section{A four-term exponential splitting scheme} \label{sec:method}

We first rewrite the Pauli equation~\eqref{eq:pauli} into a non-dimensionalized 
form\footnote{The dimensionless scaling parameter is $\eps=\frac{\hbar}{m c L_I}$ where $L_I$ is a suitable reference length. The potentials $\mathbf{A}\rightarrow \frac{\mathbf{A}}{A_I}$ and $\phi \rightarrow \frac{\phi}{\phi_I}$ are scaled by the factors $A_I = \frac{mc}{q}$ and $\phi_I=cA_I$.} :
\begin{equation} \label{eq:scaledPauli}
i \eps \, \partial_t u =
\left[ \frac{1}{2} (-i \eps \nabla - \AA)^2 - \frac{\eps}{2} \, \ssigma \cdot \BB + \phi \right]u.
\end{equation}
The rescaled magnetic field and potentials (not relabeled) are also dimensionless. For the purpose of numerics we pose the problem not in whole space $\R^3$, 
but on the space-time box $\Omega_T := \Omega \times (0,T)$,
where $\Omega := \prod_{\ell=1}^3 [0,L_\ell]$ is a rectangular cuboid, and $T>0$. We further choose periodic boundary conditions on $\Omega$ for $u(x,t)$ and a regular initial condition $u(x,0) = u^0(x)$, $x \in \Omega$,
where $u^0 \in C^{\infty}(\overline{\Omega})^2$ is periodic.

Imposing the Coulomb gauge, i.e. requiring that $\nabla \cdot \AA = 0$, and writing individually
the two equations in~\eqref{eq:scaledPauli}, we obtain the system
\begin{equation} \label{eq:pauli:rescaled}
\begin{split}
i \eps \, \partial_t u_1
& =
\left[-\frac{\eps^2}{2} \nabla^2 + i\eps \AA \cdot \nabla + \left(\frac{1}{2} \abs{\AA}^2+\phi-\frac{\eps}{2}B_3\right)\right]u_1 + \left[-\frac{\eps}{2}B_1 + \frac{i\eps}{2} B_2\right] u_2, \\
i \eps \, \partial_t u_2
& =
\left[-\frac{\eps^2}{2} \nabla^2 + i\eps \AA \cdot \nabla + \left(\frac{1}{2}\abs{\AA}^2+\phi+\frac{\eps}{2}B_3\right)\right]u_2 + \left[-\frac{\eps}{2}B_1 - \frac{i\eps}{2} B_2 \right]u_1.
\end{split}
\end{equation}
With the operators
\begin{gather*}
\mathcal{A} = \frac{i \eps}{2}\nabla^2,
\quad
\mathcal{B}_1= - \frac{i}{\eps} \left(\frac{1}{2}\abs{\AA}^2+\phi-\frac{\eps}{2}B_3\right),
\quad
\mathcal{B}_2= - \frac{i}{\eps} \left(\frac{1}{2}\abs{\AA}^2+\phi+\frac{\eps}{2}B_3\right),\\
\mathcal{C} = \AA \cdot \nabla,
\quad
\mathcal{D}_1 = \frac{i}{2}B_1 + \frac{1}{2} B_2,
\quad
\mathcal{D}_2 = \frac{i}{2}B_1 - \frac{1}{2} B_2,\\
\Afrak=\begin{pmatrix}\mathcal{A} & 0\\0 & \mathcal{A}\end{pmatrix},
\quad
\Bfrak=\begin{pmatrix}\mathcal{B}_1 & 0\\0 & \mathcal{B}_2\end{pmatrix},
\quad
\Cfrak=\begin{pmatrix}\mathcal{C} & 0\\0 & \mathcal{C}\end{pmatrix},
\quad
\Dfrak=\begin{pmatrix}0 & \mathcal{D}_1\\\mathcal{D}_2 & 0\end{pmatrix}.
\end{gather*}
We can rewrite problem~\eqref{eq:pauli:rescaled} as
\begin{equation}\label{epsilonpauli}
\partial_t u = (\Afrak + \Bfrak + \Cfrak + \Dfrak) u.
\end{equation}
Using the standard semigroup notation,
we denote
its exact solution by
\begin{equation*}
u(x,t) = \mathrm{e}^{(t-t')(\Afrak + \Bfrak + \Cfrak + \Dfrak)} u(x,t')
\quad \text{for all } x \in \Omega \text{ and } t \ge t' \ge 0.
\end{equation*}
The Pauli operator is split into four contributions:
the kinetic part ($\Afrak$), which involves the Laplace operator,
the potential part ($\Bfrak$), which collects the scalar terms of the potentials and the diagonal part of the spin term,
the advection part ($\Cfrak$), which includes the convection due the magnetic vector potential, 
and the coupling part ($\Dfrak$), peculiar of the Pauli equation,
which collects the off-diagonal part of the spin term and in general determines
the coupling of the two components of the spinor.

In view of this decomposition, the idea is to approach the time discretization of the Pauli equation
with a four-term operator splitting method in analogy with the three-term splitting method proposed
in~\cite{jin_semi-lagrangian_2013,caliari_splitting_2017,mzz2017}
for the scalar magnetic Schr\"odinger equation: Given an integer $N>0$, we consider a uniform partition of the time interval $[0,T]$
with time-step size $\timestepsize := T/N$, i.e. $t_n := n \timestepsize$ for all $n = 0, \dots, N$,
and denote by $U^n(x)$ the numerical approximation of $u(x,t_n)$.
We consider the Lie exponential splitting scheme
\begin{equation*}
U^{n+1}
= \mathrm{e}^{\timestepsize \Dfrak} \,
\mathrm{e}^{\timestepsize \Cfrak} \,
\mathrm{e}^{\timestepsize \Afrak}\, 
\mathrm{e}^{\timestepsize \Bfrak} \, U^n,
\end{equation*}
 so in the implementation, this method needs to solve each of the four steps separately
to advance the state by one time-step $\timestepsize$.
Extensions of the results in this paper to higher order splitting methods such as Strang splitting are straightforward.
For special cases, e.g. for time-independent potentials, significant computational cost
can be saved in some of the steps by pre-computing the (then analytical) solution outside of the solution
step loop for all of the intended simulation time.

For the spatial discretization of $\Omega := \prod_{\ell=1}^3 [0,L_\ell]$,
for $N_\ell \geq 2$ and $\ell=1,2,3$, let $\meshsize_\ell = L_\ell/N_\ell,$.
We define the grid size as $\abs{\meshsize}$, where $\meshsize = (\meshsize_1, \meshsize_2, \meshsize_3)$.
The set of grid points $\{x_{\jj}\}$ then consists of points
$x_{\jj} = (\frac{j_1 L_1}{N_1},\frac{j_2 L_2}{N_2},\frac{j_3 L_3}{N_3})$,
where $\jj = (j_1,j_2,j_3)$ with $0 \leq j_\ell \leq N_\ell-1$.
We denote the values of a periodic function $v:\Omega \to \C^2$ at the grid points as
\begin{equation*}
v_{j_1,j_2,j_3} := v(x _{\jj}) = v\left(\frac{j_1 L_1}{N_1},\frac{j_2 L_2}{N_2},\frac{j_3 L_3}{N_3}\right).
\end{equation*}

Some steps of the splitting scheme will be performed in Fourier space.
To that end, for a given periodic function $v:\Omega \to \C^2$,
we denote by $\widehat v_{k_1,k_2,k_3}$ its discrete Fourier transform computed via FFT, i.e.
\begin{equation*} 
\widehat v_{k_1,k_2,k_3}
= \frac{1}{N_1N_2N_3}
\sum_{j_1 =0}^{N_1-1}\sum_{j_2 =0}^{N_2-1}\sum_{j_3 =0}^{N_3-1}
 v_{j_1,j_2,j_3} \mathrm{e}^{-2\pi i \sum_{\ell=1}^3 \frac{j_\ell k_\ell} {N_\ell}}.
\end{equation*}

In the following algorithm, we summarize the structure of the proposed exponential time splitting scheme.
\begin{algorithm}[Lie splitting scheme for the Pauli equation] \label{algorithm}
\textbf{Input.} $U^0 \approx u^0$.\\
\textbf{Loop.} For each $n = 0, \dots, N-1$, iterate the following steps:
\begin{itemize}
\item[\textrm{(i)}] Potential step: Compute $U^{n*} = \mathrm{e}^{\timestepsize \Bfrak} U^n$ in physical space;
\item[\textrm{(ii)}] Kinetic step: Compute $U^{n**} = \mathrm{e}^{\timestepsize \Afrak} U^{n*}$ in Fourier space;
\item[\textrm{(iii)}] Advection step: Compute $U^{n***} = \mathrm{e}^{\timestepsize \Cfrak} U^{n**}$ in Fourier space;
\item[\textrm{(iv)}] Coupling step: Compute $U^{n+1} = \mathrm{e}^{\timestepsize \Dfrak} U^{n***}$ in physical space.
\end{itemize}
\textbf{Output.} $U^N(x_{\jj}) \approx u(x_{\jj},t_N)$.
\end{algorithm}

The numerical methods used to solve the individual ODEs as well as the order of solving the steps are in principle arbitrary but since both the kinetic and advection steps can be efficiently solved in Fourier space, some computational cost can be saved by arranging them such that only one Fourier and inverse Fourier transform step is required per time step. We include a brief discussion of each individual step and possible numerical approaches in Appendix A.

\section{Analysis of the method} \label{sec:analysis}
In this section we generalize the approach of~\cite{bjm2002, jin_semi-lagrangian_2013, mzz2017} to study the stability and convergence of the splitting scheme for the Pauli equation
described in Section~\ref{sec:method}. The results are based on using the methods suggested in Appendix A for the individual ODEs.
\subsection{Stability analysis} \label{stabilitysec}
Consider the discrete $\ell^2$ norm and the $L^2$ norm for functions given by:
\begin{align*}
\norm[\ell^2]{\UU_i}^2
&= \left( \prod_{\ell=1}^3 \frac{L_\ell}{N_\ell} \right) \sum_{j_1 =0}^{N_1-1}\sum_{j_2 =0}^{N_2-1}\sum_{j_3 =0}^{N_3-1} \abs{U_i(x_{\jj})}^2,\\
\norm[L^2]{U_i}^2
&= \int_\Omega \abs{U_i(x)}^2 \text{d} x,
\end{align*}
where $\UU_i$ denotes the vector of coefficients $U_i(x_{\jj}) = U_i\left(\frac{j_1 L_1}{N_1},\frac{j_2 L_2}{N_2},\frac{j_3 L_3}{N_3}\right)$.
The $i$ index is added here to denote that these norms are defined for spinor components as opposed to the 2-spinor itself. The total 2-spinor norm in question is the sum of the two spinor component norms.
For the sake of simplicity we assume the potentials to be time-independent,
so that analytic solutions for the potential and the coupling steps are available for all time.\\
In the following three lemmas we state three auxilliary results for the proof of stability of Algorithm \ref{algorithm}.

\begin{lemma}\label{lemma1}
Let $U_i^{n**}(x_{\jj})$ denote the elements of the grid point vector $\UU_i^{n**}$ after solving the kinetic and potential step starting from $\UU_i^{n}$. Then, it holds that
\begin{equation*}
\| \UU_i^{n**} \|_{\ell^2} = \| \UU_i^{n} \|_{\ell^2}.
\end{equation*}
and thus
\begin{equation*}
\| \UU_1^{n**} \|_{\ell^2} + \| \UU_2^{n**} \|_{\ell^2} = \| \UU_1^{n} \|_{\ell^2} + \| \UU_2^{n} \|_{\ell^2}.
\end{equation*}
\end{lemma}

\begin{proof}
The proof is a higher dimensional analogue of \cite[Lemma 3.1]{bjm2002}. We explicitly omit the $i$-index notation above in this proof despite the functions being spinor \emph{components} as opposed to the full 2-spinor to avoid excessive notational clutter.
It holds that
\begin{align*}
\| \UU^{n**} \|_{\ell^2}^2 &= \left( \prod_{\ell=1}^3 \frac{L_\ell}{N_\ell} \right) \sum_{\jj} \abs{U^{n**}(x_{\jj})}^2 = \left( \prod_{\ell=1}^3 \frac{L_\ell}{N_\ell^2} \right) \sum_{\mathbf{k}} \abs{\widehat U^{n**}_{k_1,k_2,k_3}
}^2 \\ &= 
\left( \prod_{\ell=1}^3 \frac{L_\ell}{N_\ell^2} \right) \sum_{\mathbf{k}} \abs{\mathrm{e}^{- \frac{i \eps \timestepsize}{2} \sum_{\ell=1}^3 \left( \frac{2 \pi k_\ell}{L_\ell}\right)^2} \widehat U^{n*}_{k_1,k_2,k_3}}^2 =  \left( \prod_{\ell=1}^3 \frac{L_\ell}{N_\ell^2} \right) \sum_{\mathbf{k}} \abs{\widehat U^{n*}_{k_1,k_2,k_3}}^2 \\ &= 
\left( \prod_{\ell=1}^3 \frac{L_\ell}{N_\ell} \right) \sum_{\jj} \abs{U^{n*}(x_{\jj})}^2  = \| \UU^{n*} \|_{\ell^2}^2,
\end{align*}
where we used the shorthands
$$\sum_{\jj} \rightarrow \sum_{j_1 =0}^{N_1-1}\sum_{j_2 =0}^{N_2-1}\sum_{j_3 =0}^{N_3-1} \qquad \text{and} \qquad \sum_{\mathbf{k}} \rightarrow \sum_{k_1 =0}^{N_1-1}\sum_{k_2 =0}^{N_2-1}\sum_{k_3 =0}^{N_3-1}.$$
The second and fifth step of the above computation make use of a higher dimensional variant of Plancherel's theorem (compare, e.g. \cite{kapralov_dimension-independent_2019}) which exploits a generalization of the same structure used for the one-dimensional Schrödinger variant in \cite{bjm2002}.
The potential step can be solved exactly, so the remaining statement $\| \UU^{n*} \|_{\ell^2}^2 = \| \UU^{n} \|_{\ell^2}^2$  is straightforward. Summing both spinor component results completes the proof.
\end{proof}
\begin{lemma}\label{lemma2}
Under the assumption that errors from the interpolation and backwards step are negligible
and $\nabla \cdot \mathbf{A} = 0$, the advection step solution satisfies
\begin{equation*}
\| U^{n***}_{I,i} \|_{L^2} \leq \| U^{n**}_{I,i} \|_{L^2},
\end{equation*}
where $U^{n***}_{I,i}$ denotes the Fourier interpolation of the $i$-th spinor component $U_{i}^{n***}$, and thus
\begin{equation*}
\| U^{n***}_{I,1} \|_{L^2} + \| U^{n***}_{I,2} \|_{L^2} \leq \| U^{n**}_{I,1} \|_{L^2} +\| U^{n***}_{I,2} \|_{L^2}.
\end{equation*}
\end{lemma}
\begin{proof}
See proofs of analogous results in \cite[Lemma 3.2]{mzz2017} and \cite[Theorem 1]{suli_spectral_1991} for the first statement. The second is an immediate corollary.
\end{proof}

\begin{lemma}\label{lemma3}
Let $\UU_i^{n+1}$ denote the grid point vector after solving the coupling step
starting from $\UU_i^{n***}$.
Then, it holds that
\begin{equation*}
\| \UU_1^{n+1} \|_{\ell^2} + \| \UU_2^{n+1} \|_{\ell^2} = \| \UU_1^{n***} \|_{\ell^2} + \| \UU_2^{n***} \|_{\ell^2}.
\end{equation*}
\end{lemma}

\begin{proof}
The coupling step may be solved analytically as with the potential step before and thus any analysis of this sort can be reduced to an analysis of this exact solution. However, while $e^{\timestepsize \mathfrak{D}}$ is unitary, unlike in the other cases the spinor-component-wise operators are not necessarily unitary here. Nevertheless, the stated result still holds by \emph{total} mass conservation in the Pauli equation.
\end{proof}

In the following theorem, we establish the stability of Algorithm~\ref{algorithm}.

\begin{theorem}\label{stabilitytheorem}
Let $\UU^{n+1}$ be the grid point vector after passing through all of the steps outlined in Algorithm~\ref{algorithm} once, starting from $\UU^{n}$. Then, it holds that
\begin{equation*}
\| \UU_1^{n+1} \|_{\ell^2} + \| \UU_2^{n+1} \|_{\ell^2} \leq \| \UU_1^{n} \|_{\ell^2} + \| \UU_2^{n} \|_{\ell^2}.
\end{equation*}
\end{theorem}
\begin{proof}
Applying Lemma \ref{lemma1} and \ref{lemma2} sequentially we see that
\begin{equation*}
\| \UU_i^{n***} \|_{\ell^2} = \| U^{n***}_{I,i} \|_{L^2} \leq \| U^{n**}_{I,i} \|_{L^2} = \| \UU_i^{n**} \|_{\ell^2} = \| \UU_i^{n} \|_{\ell^2},
\end{equation*}
and thus
\begin{equation*}
\| \UU_1^{n***} \|_{\ell^2} + \| \UU_2^{n***} \|_{\ell^2} \leq \| \UU_1^{n} \|_{\ell^2} + \| \UU_2^{n} \|_{\ell^2}.
\end{equation*}
The conclusion then follows via Lemma \ref{lemma3}.
\end{proof}

\subsection{Error estimates}

In this section we study error estimates for the proposed method.
For this purpose we will make use of the following shorthands to avoid overly long and repeated summation notation:
\begin{align*}
\norm[\alpha]{\UU} &= \norm[\ell^2]{\UU_1} + \norm[\ell^2]{\UU_2},\\
\norm[A]{u} &= \norm[L^2]{u_1} + \norm[L^2]{u_2}.
\end{align*}
The main result of this section will be using the following assumptions, which are  analogues of the assumptions for the scalar Schrödinger-type equation in \cite{bjm2002, jin_semi-lagrangian_2013, mzz2017}: We assume the solutions and potentials are smooth and periodic on the spatial box.
With $m \geq 1$ let
\begin{align}\label{assumptions1}
\left\lVert{\frac{\partial^{m_1}}{\partial x^{m_1}} \frac{\partial^{m_2}}{\partial t^{m_2}}  u_i(x,t)}\right\rVert_{C([0,T],L^2)} &\leq \frac{\alpha_{m_1+m_2}}{\eps^{m_1+m_2}}, \\
\left\lVert{\frac{\partial^m}{\partial x^m} \mathbf{A}(x)}\right\rVert_{L^2} \leq \beta_m &,\label{assumptions3} \quad \left\lVert{\frac{\partial^m}{\partial x^m} \phi(x)}\right\rVert_{L^2} \leq \gamma_m,
\end{align}
where $m,m_1,m_2 \in \mathbb{N}$ and $m = m_1+m_2$,
while $\alpha_m, \beta_m$ and $\gamma_m$ are positive constants and $\eps$ is the (small) scaling parameter appearing in the scaled Pauli equation.
Wherever we write $u_i$ without specific index $i$ we mean to imply that the statement holds for both of the 2-spinor components individually and thus has an obvious extension to the $\alpha$ and $A$ norms defined above.

\begin{theorem}\label{convergencetheorem}
Denote the exact 2-spinor solution to the Pauli equation in~\eqref{epsilonpauli} for given parameter $\eps$ by $u^{\eps}(x,t) = \begin{pmatrix} u_1^{\eps}(x,t) \\ u_2^{\eps}(x,t) \end{pmatrix}$,
where
\begin{equation*}
u^{\eps}(x,t+\timestepsize)
= \mathrm{e}^{\timestepsize\Afrak + \timestepsize\Bfrak + \timestepsize\Cfrak + \timestepsize\Dfrak} u^{\eps}(x,t),
\end{equation*}
and its operator splitting numerical approximation at time $n\timestepsize$ by $U^{n}$, where
\begin{equation*}
U^{n+1}
= \mathrm{e}^{\timestepsize\Dfrak} \mathrm{e}^{\timestepsize\Cfrak} \mathrm{e}^{\timestepsize\Afrak} \mathrm{e}^{\timestepsize\Bfrak}  U^{n}.
\end{equation*}
We assume that the potentials and solution are smooth and periodic on the relevant spatial box,
that the characteristic equation in~\eqref{characteristic} in the advection step and the FFT steps may be solved with negligible error
along with the assumption statements listed in~\eqref{assumptions1}--\eqref{assumptions3}
and that $\abs{\meshsize} = O(\eps)$ and $\timestepsize = O(\eps)$.
Then, for any time $t \in [0, T]$ we have the error estimate
\begin{align*}
\| u^{\eps}(t_{n})-U_I^{n}\|_A &\leq \frac{C_1T}{\timestepsize}\left(\frac{\abs{\meshsize}}{\eps}\right)^m + \frac{C_2 T \timestepsize}{\eps},
\end{align*}
with $C_1, C_2$ being constants independent of $\timestepsize$, $\meshsize$, $T$, and $\eps$.
\end{theorem}

\begin{proof}
Similar to discussions in \cite[Theorem 4.1]{bjm2002}, \cite[Theorem 4]{jin_semi-lagrangian_2013} and \cite[Theorem 3.2]{mzz2017} for various cases of scalar Schrödinger-type equations, the local splitting error for the Pauli equation operator splitting method is also determined by the non-commutativity of the respective operators via the classical Baker– Campbell–Hausdorff formula.
The proof strategy thus begins with the computation of commutators $[\cdot,\cdot]$ for the operators
$\timestepsize\Afrak$, $\timestepsize\Bfrak$, $\timestepsize\Cfrak$ and $\timestepsize\Dfrak$ and then concludes via a triangle inequality estimation for the error and can thus be seen as a 2-spinor generalization of the above referenced theorems.\\
As the operators in question act on 2-spinors and have a block operator representation we make use of the observation that the commutators of such operators $\mathfrak{L},\mathfrak{M},\mathfrak{K}$ with form
\begin{align*}
\mathfrak{L} = \begin{pmatrix}L & 0\\0 & L\end{pmatrix},\quad
\mathfrak{M} = \begin{pmatrix}M_1 & 0\\0 & M_2\end{pmatrix},\quad
\mathfrak{K} = \begin{pmatrix}0 & K_1\\K_2 & 0\end{pmatrix}
\end{align*}
 satisfy
\begin{align*}
[\mathfrak{L},\mathfrak{M}]&=\begin{pmatrix}[L,M_1] & 0\\0 & [L,M_2]\end{pmatrix}, \quad [\mathfrak{L},\mathfrak{K}]=\begin{pmatrix}0 & [L,K_1]\\ [L,K_2] & 0\end{pmatrix}, \\
[\mathfrak{M},\mathfrak{K}]&=\begin{pmatrix}0 & M_1 K_1 - K_1 M_2\\ M_2K_2-K_2M_1 & 0\end{pmatrix}.
\end{align*}
The computation can thus be made easier by computing these for each of the relevant component operators of $\timestepsize\Afrak$, $\timestepsize\Bfrak$, $\timestepsize\Cfrak$ and $\timestepsize\Dfrak$.
Direct computation yields the following results for the non-coupling commutators:
\begin{align*}
[\timestepsize \mathcal{A}, \timestepsize \mathcal{B}_1] u_1
&=  \frac{(\timestepsize)^2}{2}\sum_{j=1}^3\partial_j^2 \left(\frac{1}{2}\abs{\AA}^2+\phi-\frac{\eps}{2}B_3 \right)u_1 \\
& \quad + (\timestepsize)^2 \sum_{j=1}^3 \partial_{j} \left(\frac{1}{2}\abs{\AA}^2+\phi-\frac{\eps}{2}B_3 \right)\partial_{j} u_1,\\
[\timestepsize\mathcal{A}, \timestepsize \mathcal{B}_2]u_2
&= \frac{(\timestepsize)^2}{2}\sum_{j=1}^3\partial_j^2 \left(\frac{1}{2}\abs{\AA}^2+\phi+\frac{\eps}{2}B_3 \right)u_2\\
& \quad + (\timestepsize)^2 \sum_{j=1}^3 \partial_{j} \left(\frac{1}{2}\abs{\AA}^2+\phi+\frac{\eps}{2}B_3 \right)\partial_{j} u_2,
\end{align*}
\begin{align*}
[\timestepsize \mathcal{A}, \timestepsize \mathcal{C}] u_i
&= \frac{i \eps (\timestepsize)^2}{2} \sum_{k=1}^3 \sum_{j=1}^3 \left( \partial_k^2 A_j \partial_j u_i +2 \partial_k A_j \partial_k \partial_j u_i \right),\\
[\timestepsize \mathcal{C}, \timestepsize \mathcal{B}_1] u_1
&= - \frac{i(\timestepsize)^2}{\eps} \sum_{j=1}^3 A_j \partial_j  \left(\frac{1}{2}\abs{\AA}^2+\phi-\frac{\eps}{2}B_3 \right)u_1,\\
[\timestepsize \mathcal{C}, \timestepsize \mathcal{B}_2] u_2
&= - \frac{i(\timestepsize)^2}{\eps} \sum_{j=1}^3 A_j \partial_j  \left(\frac{1}{2}\abs{\AA}^2+\phi+\frac{\eps}{2}B_3 \right)u_2.
\end{align*}
This covers the operators which are already present in the magnetic Schr\"odinger case.
The primary takeaway from this is that the worst case scenario for the error from these commutators is of order $O\left(\frac{(\timestepsize)^2}{\eps}\right)$, consistent with \cite{bjm2002, jin_semi-lagrangian_2013, mzz2017}. This even holds true for the $\mathfrak{B}$ operator which differs from the magnetic Schrödinger case. Direct computation of the components of the coupling step commutators yields: 
\begin{align*}
[\timestepsize \mathcal{A}, \timestepsize \mathcal{D}_1] u_2
&=  \frac{(\timestepsize)^2}{2}\sum_{j=1}^3\partial_j^2  \left(-\frac{\eps}{2}B_1 + \frac{i \eps}{2} B_2 \right)u_1 \\
& \quad + (\timestepsize)^2 \sum_{j=1}^3 \partial_{j} \left(-\frac{\eps}{2}B_1 + \frac{i \eps}{2} B_2 \right)\partial_{j} u_2,\\
[\timestepsize \mathcal{A}, \timestepsize \mathcal{D}_2] u_1
&=  \frac{(\timestepsize)^2}{2}\sum_{j=1}^3\partial_j^2  \left(-\frac{\eps}{2}B_1 - \frac{i \eps}{2} B_2 \right)u_1 \\
& \quad + (\timestepsize)^2 \sum_{j=1}^3 \partial_{j} \left(-\frac{\eps}{2}B_1 - \frac{i \eps}{2} B_2 \right)\partial_{j} u_1,\\
[\timestepsize \mathcal{C}, \timestepsize \mathcal{D}_1] u_2
&= (\timestepsize)^2 \sum_{j=1}^3 A_j \partial_j \left(\frac{i}{2} B_1 + \frac{1}{2}B_2\right)u_2,\\
[\timestepsize \mathcal{C}, \timestepsize \mathcal{D}_2] u_1
&= (\timestepsize)^2 \sum_{j=1}^3 A_j \partial_j \left(\frac{i}{2}B_1 - \frac{1}{2} B_2\right)u_1,\\
\timestepsize^2 (\mathcal{B}_1 \mathcal{D}_1-\mathcal{D}_1 \mathcal{B}_2)u_2
&= \frac{(\timestepsize)^2}{2} (-B_1 B_3 + i B_2 B_3),\\
\timestepsize^2 (\mathcal{B}_2 \mathcal{D}_2-\mathcal{D}_2 \mathcal{B}_1)u_1
&= \frac{(\timestepsize)^2}{2} (B_1 B_3 + i B_2 B_3).
\end{align*}
All of these terms are $O\left((\timestepsize)^2\right)$.
As $\timestepsize = O (\eps)$ this means the coupling step commutators contribute less to the error than the previously mentioned worst case scenario.
Combining these results for all of the commutators, one finds that the local splitting error satisfies
\begin{equation*}
\| u^{\eps}(t_{n+1})-\tilde{u}(t_{n+1})\|_A = O\left(\frac{\timestepsize^2}{\eps}\right),
\end{equation*}
where $\tilde{u}(t_{n+1})$ is the pre-discretization operator splitting solution satisfying
\begin{equation*}
\tilde{u}(t_{n+1}) = \mathrm{e}^{\timestepsize\Dfrak} \mathrm{e}^{\timestepsize\Cfrak} \mathrm{e}^{\timestepsize\Afrak} \mathrm{e}^{\timestepsize\Bfrak} u(t_n).
\end{equation*}
Due to the nature of the coupling step,
the errors in the two spin components are not in general separable.
We proceed via the triangle inequality as follows:
\begin{align*}
\| u^{\eps}(t_{n+1})-U_I^{n+1}\|_A
&\leq \| u^{\eps}(t_{n+1})-\tilde{u}(t_{n+1})\|_A + \| \tilde{u}(t_{n+1}) - \tilde{u}_I(t_{n+1})\|_A\\
&+ \| \tilde{u}_I(t_{n+1})- U_I^{n+1}\|_A.
\end{align*}
The first term on the right hand side was already shown above to be of order $O\left(\frac{\timestepsize^2}{\eps}\right)$, while the second term is the error of the used interpolation method which as discussed in \cite{bjm2002,jin_semi-lagrangian_2013} and \cite[Theorem 3]{pasciak_spectral_1980} is $O\left(\left(\frac{\abs{\meshsize}}{\eps}\right)^m\right)$ under the assumption in \eqref{assumptions1},
where $m$ is any positive integer.
The final term in need of investigation is thus $\| \tilde{u}(t_{n+1})- U_I^{n+1}\|_A$
which corresponds to the error incurred due to the discretization.
Noting that $\|f_I\|_{L^2} = \| \mathbf{f} \|_{\ell^2}$, we obtain
\begin{align*}
\| \tilde{u}_I(t_{n+1})- U_I^{n+1}\|_A
&= \| \tilde{\mathbf{u}}(t_{n+1})- \mathbf{U}^{n+1} \|_\alpha\\
&=\|\mathrm{e}^{\timestepsize\Dfrak} \mathrm{e}^{\timestepsize\Cfrak} \mathrm{e}^{\timestepsize\Afrak} \mathrm{e}^{\timestepsize\Bfrak}  \mathbf{u}(t_{n}) - \mathrm{e}^{\timestepsize\Dfrak} \mathrm{e}^{\timestepsize\Cfrak_N} \mathrm{e}^{\timestepsize\Afrak_N} \mathrm{e}^{\timestepsize\Bfrak}  \mathbf{U}^{n} \|_\alpha,
\end{align*}
where $\tilde{\mathbf{u}}$ and $\mathbf{u}$ denote the vectors collecting the gridpoint values
of $\tilde{u}$ and $u$, respectively.
As the potential and coupling steps are solved analytically the operators remain unaffected on the right hand side but for the kinetic and advection steps we must distinguish
their numerical approximations $\mathfrak{A}_N$ and $\mathfrak{C}_N$.
A further application of the triangle inequality yields
\begin{align*}
\| \tilde{u}_I(t_{n+1})- U_I^{n+1}\|_A &\leq \|\mathrm{e}^{\timestepsize\Dfrak} \mathrm{e}^{\timestepsize\Cfrak} \mathrm{e}^{\timestepsize\Afrak} \mathrm{e}^{\timestepsize\Bfrak}  \mathbf{u}(t_{n}) - \mathrm{e}^{\timestepsize\Dfrak} \mathrm{e}^{\timestepsize\Cfrak} \mathrm{e}^{\timestepsize\Afrak_N} \mathrm{e}^{\timestepsize\Bfrak}  \mathbf{u}(t_{n}) \|_\alpha\\
&+ \|\mathrm{e}^{\timestepsize\Dfrak} \mathrm{e}^{\timestepsize\Cfrak} \mathrm{e}^{\timestepsize\Afrak_N} \mathrm{e}^{\timestepsize\Bfrak}  \mathbf{u}(t_{n}) - \mathrm{e}^{\timestepsize\Dfrak} \mathrm{e}^{\timestepsize\Cfrak_N} \mathrm{e}^{\timestepsize\Afrak_N} \mathrm{e}^{\timestepsize\Bfrak}  \mathbf{u}(t_{n})\|_\alpha\\
&+ \|\mathrm{e}^{\timestepsize\Dfrak} \mathrm{e}^{\timestepsize\Cfrak_N} \mathrm{e}^{\timestepsize\Afrak_N} \mathrm{e}^{\timestepsize\Bfrak} \mathbf{u}(t_{n}) - \mathrm{e}^{\timestepsize\Dfrak} \mathrm{e}^{\timestepsize\Cfrak_N} \mathrm{e}^{\timestepsize\Afrak_N} \mathrm{e}^{\timestepsize\Bfrak}  \mathbf{U}^{n} \|_\alpha.
\end{align*}
The first term in the above is just a measure of the spectral approximation error again and is thus $O\left(\left(\frac{\abs{\meshsize}}{\eps}\right)^m\right)$ as described above. The same is true for the second term since if errors due to the computation of the backwards grid step are negligible then this step is just a measure for the interpolation accuracy.
For the final term we note that since the operators $\mathrm{e}^{\timestepsize\Dfrak}, \mathrm{e}^{\timestepsize\Cfrak}, \mathrm{e}^{\timestepsize\Afrak}$ and $\mathrm{e}^{\timestepsize\Bfrak}$ are all unitary, the lemmas leading up to Theorem \ref{stabilitytheorem} in particular also imply that
\begin{align*}
\|\mathrm{e}^{\timestepsize\Cfrak_N}\|_A \leq 1, \quad \|\mathrm{e}^{\timestepsize\Afrak_N}\|_A = 1.
\end{align*}
Using the stability results above, cf. \cite[Equation~(3.50)]{mzz2017}, then yields
\begin{align*}
\|\mathrm{e}^{\timestepsize\Dfrak} \mathrm{e}^{\timestepsize\Cfrak_N} \mathrm{e}^{\timestepsize\Afrak_N} \mathrm{e}^{\timestepsize\Bfrak} \mathbf{u}(t_{n}) - \mathrm{e}^{\timestepsize\Dfrak} \mathrm{e}^{\timestepsize\Cfrak_N} \mathrm{e}^{\timestepsize\Afrak_N} \mathrm{e}^{\timestepsize\Bfrak}  \mathbf{U}^{n} \|_\alpha \leq \| u^{\eps}(t_n) -  U_I^{n}\|_A.
\end{align*}
Analogous to \cite[equation~(3.52)]{mzz2017} the discussion so far yields a recursive relationship for error accumulation
\begin{equation*}
\| u^{\eps}(t_{n+1})-U_I^{n+1}\|_A
\leq \| u^{\eps}(t_{n})-U_I^{n}\|_A + C_1\left(\frac{\abs{\meshsize}}{\eps}\right)^m
+ C_2\left(\frac{\timestepsize ^2}{\eps}\right).
\end{equation*}
which on the solution interval $t \in [0,T]$ implies that
\begin{equation*}
\| u^{\eps}(t_{n})-U_I^{n}\|_A
\leq \frac{C_1T}{\timestepsize}\left(\frac{\abs{\meshsize}}{\eps}\right)^m + \frac{C_2 T \timestepsize}{\eps},
\end{equation*}
for some constants $C_1$ and $C_2$ independent of $\timestepsize$, $\meshsize$, $T$, and $\eps$.
\end{proof}

The above theorem is fully consistent with the view of the Pauli equation as a bottom-up generalization of the scalar magnetic Schrödinger equation as it yields analogous error bounds despite the inclusions of the coupling step as well as its inherent three-dimensional nature.
Furthermore we can use the above result to define a meshing strategy for a desired accuracy (as done for the magnetic Schr\"odinger equation in~\cite{mzz2017} and~\cite{jin_semi-lagrangian_2013}):
If $\delta >0$ is a desired  error bound so that  $\| u^{\eps}(t_{n})-U_I^{\eps,n}\|_\beta \leq \delta$, then one should choose $\timestepsize$ and $\meshsize$ to satisfy
\begin{align*}
\frac{\timestepsize}{\eps} = O(\delta), \quad
\left(\frac{\abs{\meshsize}}{\eps}\right)^m = O(\delta \timestepsize).
\end{align*}
\begin{remark}
We note that for solutions and fields with sufficient regularity, the second term of the error bound in Theorem \ref{convergencetheorem} dominates the error and one can thus expect approximately linear convergence in $\Delta t$. Higher order in $\Delta t$ methods can be derived in the straightforward way by replacing the Lie splitting with higher order Strang splitting schemes.
\end{remark}
\section{Numerical experiments} \label{sec:numerics}

We present proof-of-concept numerical results obtained from an implementation of the proposed method as a first order Lie splitting scheme. Higher precision may be obtained than is illustrated in this section by decreasing the stepsize in space and time, as well as using a second or higher order Strang splitting, cf. \cite{thalhammer_high-order_2008,AuzKT2015}. The computations presented in this section have been performed with an implementation of the above method
in the Julia programming language~\cite{beks2017}. We consider two cases with different spin coupling behavior and set $\Omega = [0,10]^3$, $\Delta x = 0.4$ and $\eps=0.5$ for both numerical experiments.
\subsection{Decoupled spin state dynamics}
We seek numerical solutions of the Pauli equation~\eqref{eq:pauli} using the following constant-in-time fields,
which are periodic on $\Omega$:
\begin{align} 
&\AA(x) = \pi
\begin{pmatrix}
-\cos\left(\frac{\pi}{5}(x_2-5)\right)\sin\left(\frac{\pi}{5}(x_2-5)\right)\\
\cos\left(\frac{\pi}{5}(x_1-5)\right)\sin\left(\frac{\pi}{5}(x_1-5)\right)\\
0
\end{pmatrix},\label{eq:experiment1afield}\\
\quad
&\BB(x)
= \frac{\pi}{5} 
\begin{pmatrix}
0 \\
0 \\
\sum_{j=1}^{2}\left( \pi \cos\left(\frac{\pi}{5}(x_j-5) \right)^2- \pi \sin\left(\frac{\pi}{5}(x_j-5) \right)^2\right) \end{pmatrix},\\
&\phi (x) = 0.
\end{align}
It is easily confirmed that these fields satisfy $\BB = \nabla \times \AA$
as well as the Coulomb gauge $\nabla \cdot \AA = 0$. We initialize the state
\begin{equation} \label{eq:initialstateup}
u^0
=
\begin{pmatrix}
\mathrm{e}^{-(x_1-4.5)^2-(x_2-4.5)^2-(x_3-5)^2}\\
\mathrm{e}^{-(x_1-5.5)^2-(x_2-5.5)^2-(x_3-5)^2}
\end{pmatrix}.
\end{equation}
As the $\BB$ field lacks an $x$-component, there is no coupling between spin up and down components and the spin components evolve fully independently indefinitely. In the absence of analytic solutions we can compare the obtained solutions with a more precise numerical solution to approximately visualize the method's convergence properties. Figure \ref{fig:errors1} shows the maximal absolute and relative errors, that is
\begin{align*}
Err_{abs}(u,T) = \max_{k_1,k_2,k_3} |(u^N_{k_1,k_2,k_3} - \tilde{u}^N_{k_1,k_2,k_3}|,\\
Err_{rel}(u,T) = \max_{k_1,k_2,k_3}  \frac{|(u^N_{k_1,k_2,k_3} - \tilde{u}^N_{k_1,k_2,k_3}|}{|\tilde{u}^N_{k_1,k_2,k_3}|},
\end{align*}
compared to a high precision numerical approximation $\tilde{u}$. Convergence to the approximation is approximately linear as expected of a first order Lie splitting approach with sufficiently small $\Delta x$.
  \begin{figure}
    \centering
        \subfloat[]
    {{  \includegraphics[width=6.7cm]{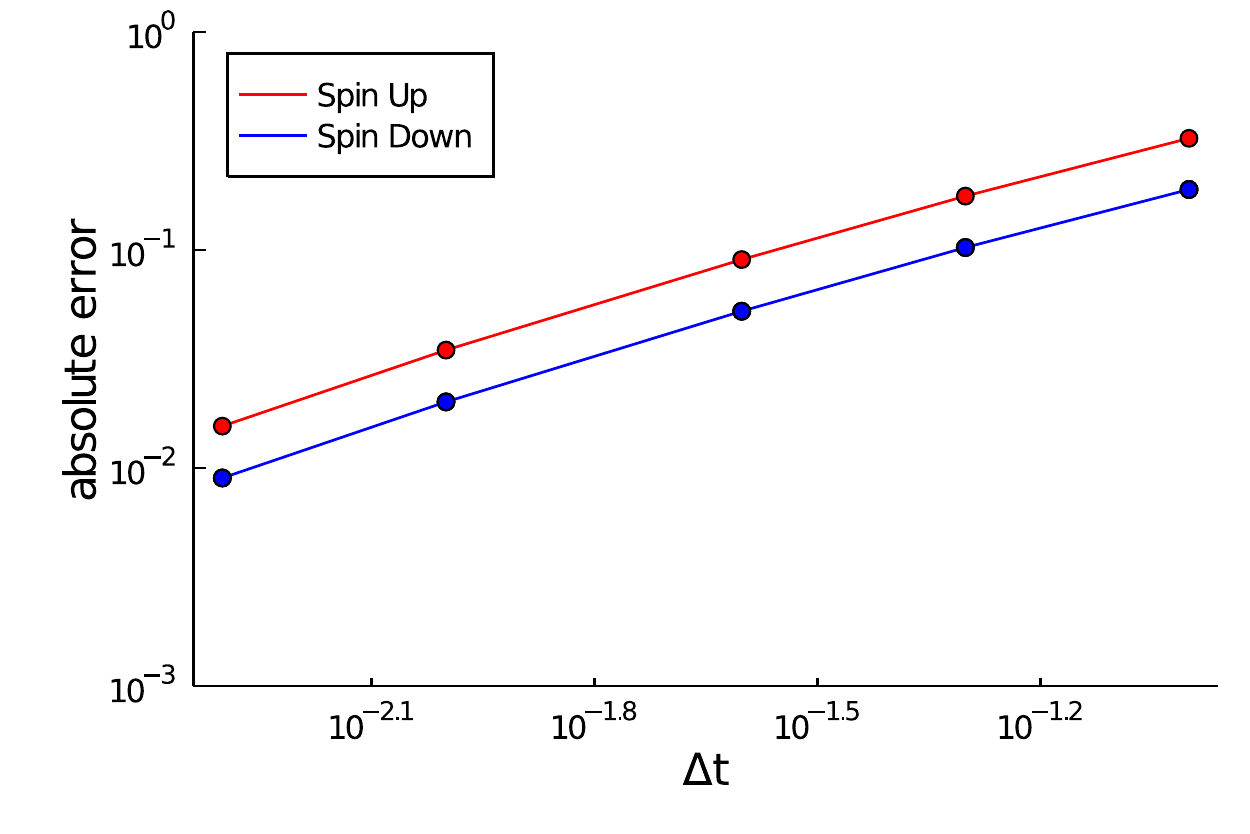} }}
   \qquad
         \subfloat[]
    {{  \includegraphics[width=6.7cm]{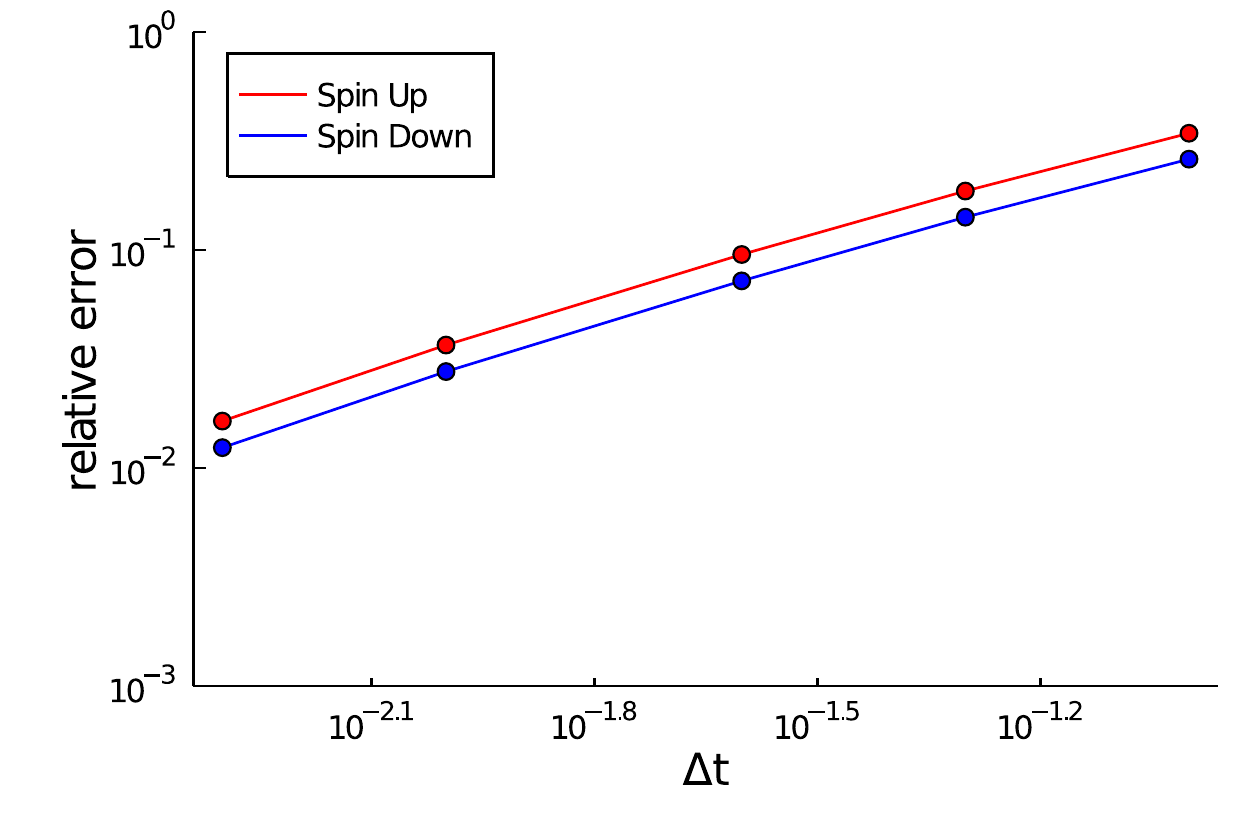} }}
    \caption{Absolute and relative errors of different time discretizations for the decoupled spin dynamics numerical experiment compared to a higher precision numerical solution. We observe approximately linear error convergence in time, as expected from a first order Lie splitting method.}%
    \label{fig:errors1}%
\end{figure}
\subsection{Coupled spin state dynamics}
We use the following modified field setup
to observe more complex Pauli equation phenomena:
\begin{align} 
&\AA(x) = \pi
\begin{pmatrix}
-\cos\left(\frac{\pi}{5}(x_2-5)\right)\sin\left(\frac{\pi}{5}(x_2-5)\right)\\
\cos\left(\frac{\pi}{5}(x_1-5)\right)\sin\left(\frac{\pi}{5}(x_1-5)\right)\\
\frac{1}{\pi} \cos\left(\frac{\pi}{5}(x_1-5)\right)\sin\left(\frac{\pi}{5}(x_2-5)\right)
\end{pmatrix},\label{eq:experiment2afield}\\
\quad
&\BB(x)
= \frac{\pi}{5} 
\begin{pmatrix}
\cos\left(\frac{\pi}{5}(x_1-5)\right)\cos\left(\frac{\pi}{5}(x_2-5)\right) \\
\sin\left(\frac{\pi}{5}(x_1-5)\right)\sin\left(\frac{\pi}{5}(x_2-5)\right) \\
\sum_{j=1}^{2} \left( \pi\cos\left(\frac{\pi}{5}(x_j-5) \right)^2-  \pi\sin\left(\frac{\pi}{5}(x_j-5) \right)^2\right) \end{pmatrix},\\
&\phi (x) = 0.
\end{align}
We initialize with an exclusively spin up state:
\begin{equation} \label{eq:initialstateup2}
u^0
=
\begin{pmatrix}
\mathrm{e}^{-(x_1-4.5)^2-(x_2-4.5)^2-(x_3-5)^2}\\
0
\end{pmatrix}.
\end{equation}
Figure~\ref{fig:exp2solutions} shows the absolute value of the solution obtained for the initial state~\eqref{eq:initialstateup2}
at different times visualized as isosurfaces.
Due to the presence of a non-zero $x_1$-component in the $\BB$ field,
one observes coupling between spin up and down components. Figure \ref{fig:errors2} shows absolute and relative errors compared to a higher precision numerical approximation.
 \begin{figure}
    \centering
    \subfloat[$t = 0.0 $]
    {{  \includegraphics[width=4cm]{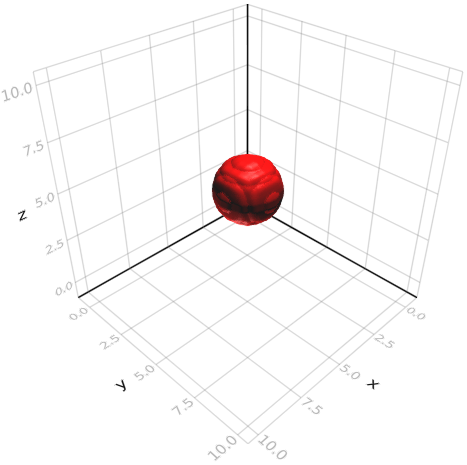} }}
   \qquad
      \subfloat[$t = 0.5 $]
    {{  \includegraphics[width=4cm]{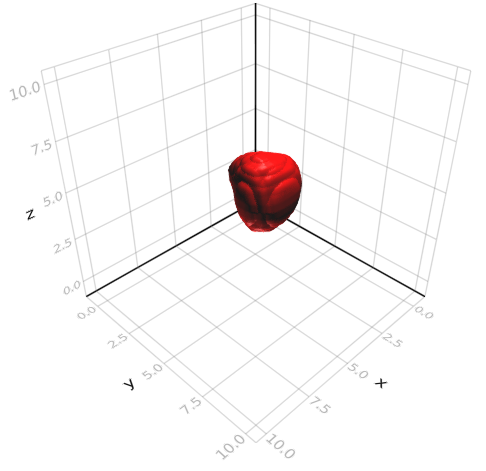} }}
               \qquad
      \subfloat[$t = 1.0 $]
    {{  \includegraphics[width=4cm]{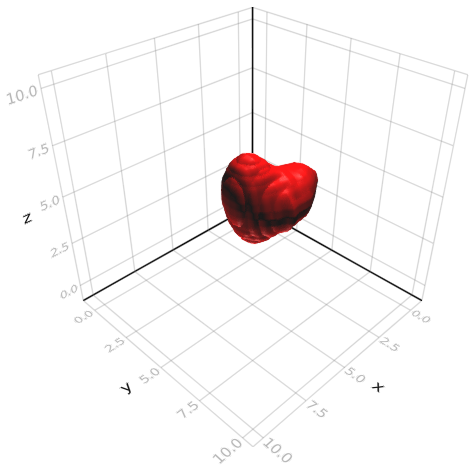} }}
    \qquad
        \subfloat[$t = 0.0 $]
    {{  \includegraphics[width=4cm]{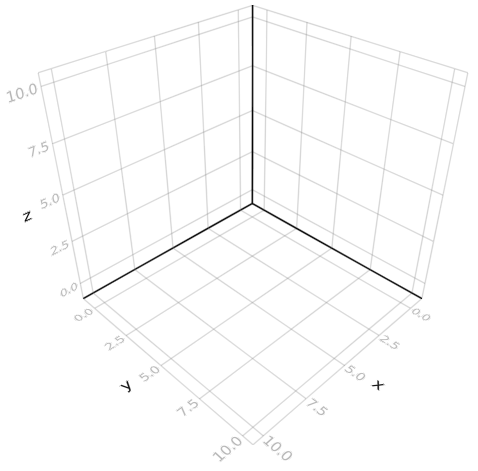} }}
   \qquad
      \subfloat[$t = 0.5 $]
    {{  \includegraphics[width=4cm]{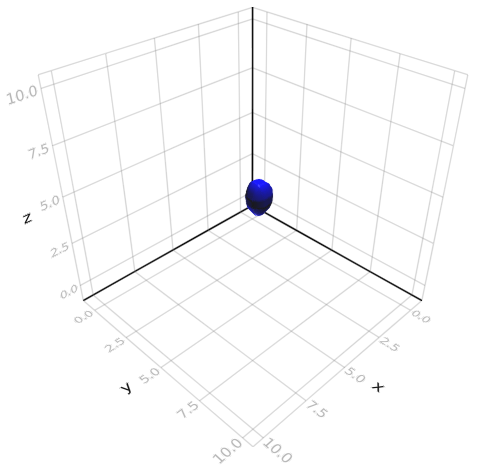} }}
               \qquad
      \subfloat[$t = 1.0 $]
    {{  \includegraphics[width=4cm]{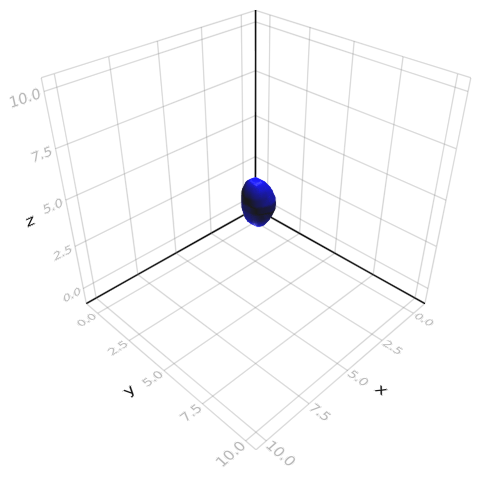} }}
    \caption{Isosurfaces of constant value $0.055$ for fields in~\eqref{eq:experiment2afield}
and initial state in~\eqref{eq:initialstateup2}.
Absolute value of spin up component displayed in red and absolute value of spin down component in blue. The spin up component gradually induces a spin down component in the same location and vice versa. The coupling is bidirectional.}%
    \label{fig:exp2solutions}%
\end{figure}
 \begin{figure}
    \centering
    \subfloat[]
    {{  \includegraphics[width=6.7cm]{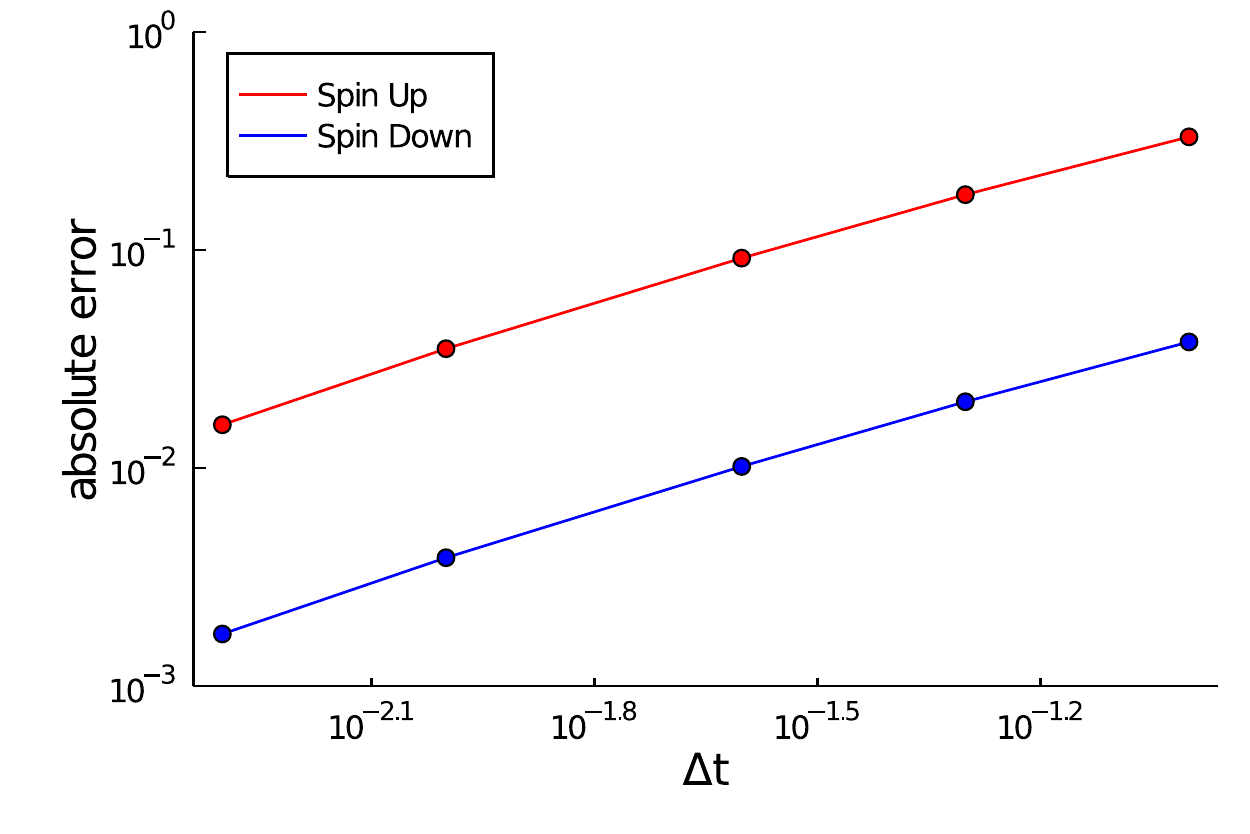} }}
   \qquad
      \subfloat[]
    {{  \includegraphics[width=6.7cm]{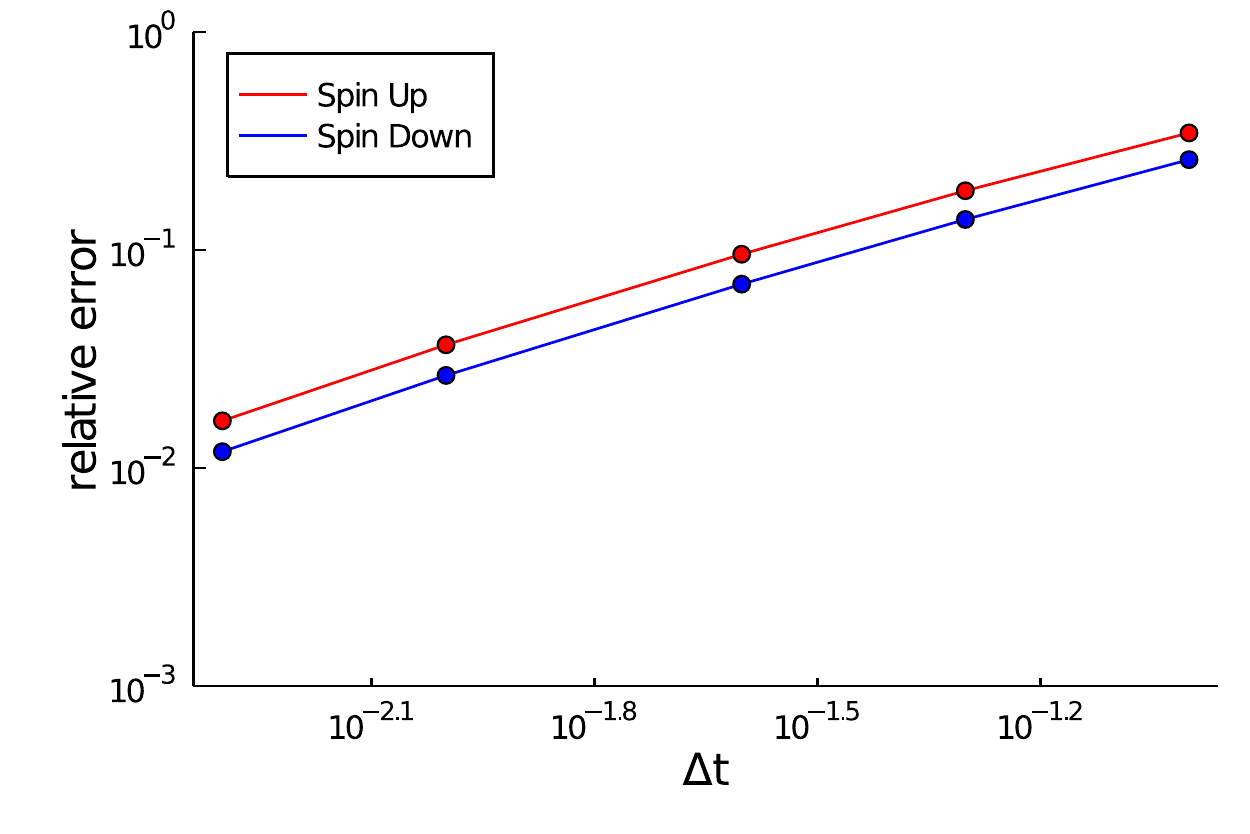} }}
    \caption{Absolute and relative errors of different time discretizations for the coupled spin system numerical experiment compared to a higher precision numerical solution. We observe approximately linear error convergence in time, as expected from a first order Lie splitting method.}%
    \label{fig:errors2}%
\end{figure}
\section{Conclusion} \label{sec:conclusion}
We extended schemes for the scalar magnetic Schr\"odinger equation without spin term~\cite{jin_semi-lagrangian_2013,caliari_splitting_2017,mzz2017} to the Pauli equation by proposing a four-term operator splitting method. We analyzed the convergence of the scheme and presented proof of concept numerical experiments. The results are applicable to time-independent as well as simple time-dependent magnetic fields, but as of now are restricted to the linear case, i.e. explicitly given external magnetic vector and scalar electric potentials with or without time-dependence.

In the numerics of the Pauli equation, the coupled nature of the spin up and spin down state equations means that any error bounds can only be valid for the sum of the two states, as any errors can and will propagate between spin up and spin down state solutions in each step. Given this fact, it is remarkable that numerical error bounds obtained for the linear Pauli equation appear well-behaved under mild assumptions.

An important question for applications is the extension of this method to the fully self-consistent system consisting of the Pauli equation coupled to a suitable first order $O\left(\frac{1}{c}\right)$ approximation of the Maxwell equations. The canonical choice would be the so-called Pauli--Poiswell system~\cite{masmoudi_selfconsistent_2001}. This is part of ongoing research on numerical methods for nonlinear Pauli equations.

\section*{Acknowledgments}

We acknowledge support of the Austrian Science Fund (FWF) via the grants FWF DK W1245 and SFB F65, support from the Vienna Science and Technology Fund (WWTF) project MA16-066 "SEQUEX".

\section*{Appendix A}

\subsection*{The potential step}
Step~(i) of Algorithm~\ref{algorithm} consists in finding, for all grid points $x_{\jj}$,
the solution to the initial value problem
\begin{alignat*}{2}
\partial_s \begin{pmatrix}w_{\jj1}(s) \\ w_{\jj2}(s) \end{pmatrix}
& =
\begin{pmatrix}
\Bfrak_1(s) & 0\\
0 & \Bfrak_2(s)
\end{pmatrix}
\begin{pmatrix} w_{\jj1}(s) \\ w_{\jj2}(s) \end{pmatrix}
&\quad& \text{for } s \in (0,\timestepsize), \\
w_{\jj}(0) &= U^n(x_{\jj}), &&
\end{alignat*}
where $w_{\jj}=(w_{\jj1},w_{\jj2})$ and
\begin{align*}
\Bfrak_1(s) & = - \frac{i}{\eps}\left( \frac{1}{2}\abs{\AA(x_{\jj},t_n+s)}^2 + \phi(x_{\jj},t_n+s) - \frac{\eps}{2}B_3(x_{\jj},t_n+s) \right),\\
\Bfrak_2(s) & = - \frac{i}{\eps}\left( \frac{1}{2}\abs{\AA(x_{\jj},t_n+s)}^2 + \phi(x_{\jj},t_n+s) + \frac{\eps}{2}B_3(x_{\jj},t_n+s) \right).
\end{align*}
Then, the solution of the potential step is given by
$U^{n*}(x_{\jj}) = \mathrm{e}^{\timestepsize \Bfrak} U^n(x_{\jj}) := w_{\jj}(\timestepsize)$.
For time-independent magnetic field and potentials, an analytical solution is available
for all time-steps outside of the solution loop,
whereas for time-dependent data the solution has to be re-computed in each time-step.
In the latter case the solution can be obtained with any highly efficient ODE solver.

\subsection*{The kinetic step}
In Step~(ii) of Algorithm~\ref{algorithm}, one has to solve the initial boundary value problem
\begin{alignat*}{2}
\partial_t \begin{pmatrix}w_1 \\ w_2 \end{pmatrix}
& =
\begin{pmatrix}\frac{i \eps}{2}\nabla^2 & 0\\0 & \frac{i \eps}{2}\nabla^2\end{pmatrix}
\begin{pmatrix}w_1 \\ w_2 \end{pmatrix}
&\quad& \text{in } \Omega \times (t_n,t_{n+1}), \\
w(t_n) &= U^{n*}, &&
\end{alignat*}
which consists of nothing but two decoupled free Schr\"odinger equations
with periodic boundary conditions for $w=(w_1,w_2)$.
Then, the solution of the kinetic step is given by $U^{n**} = \mathrm{e}^{\timestepsize \Afrak} U^{n*} := w(t_{n+1})$.
Hence, we can use any of the available highly efficient methods for the free Schr\"odinger equation.
In light of the advection step, a good way is to solve the equation in Fourier space using FFT.
In particular, as $U^{n**} = \mathrm{e}^{\timestepsize \Afrak} U^{n*}$, we find that
\begin{equation}\label{eq:kinetic_fourier}
\begin{split}
\widehat U^{n**}_{k_1,k_2,k_3}
&=\mathrm{e}^{- \frac{i \eps \timestepsize}{2} \sum_{\ell=1}^3 \left( \frac{2 \pi k_\ell}{L_\ell}\right)^2} \widehat U^{n*}_{k_1,k_2,k_3}\\
&=\mathrm{e}^{- \frac{i \eps \timestepsize}{2} \sum_{\ell=1}^3 \left( \frac{2 \pi k_\ell}{L_\ell}\right)^2}  \frac{1}{N_1N_2N_3} 
\sum_{j_1 =0}^{N_1-1}\sum_{j_2 =0}^{N_2-1}\sum_{j_3 =0}^{N_3-1}
U^{n*}_{j_1,j_2,j_3} \mathrm{e}^{-2\pi i \sum_{\ell=1}^3 \frac{j_\ell k_\ell} {N_\ell}}.
\end{split}
\end{equation}
Then, instead of performing an iFFT to move back to physical space we can directly pass
the Fourier space data to the next step.

\subsection*{The advection step}
This substep  is the most subtle step of the  operator splitting method,
as standard methods are usually stable only under restrictive CFL-type conditions 
that prevent the use of large time-step sizes.
However, since it is analogous to the magnetic Schr\"odinger equation case, we can adapt methods
in~\cite{jin_semi-lagrangian_2013,caliari_splitting_2017,mzz2017} for the 2-spinor case.
We opt for the method of characteristics to solve this equation combined with Fourier interpolation.
Step~(iii) of Algorithm~\ref{algorithm} consists of the solution of
\begin{alignat*}{2}
\partial_t \begin{pmatrix}w_1\\ w_2\end{pmatrix}
&=
\begin{pmatrix}
\AA \cdot \nabla & 0\\0 & \AA \cdot \nabla
\end{pmatrix}
\begin{pmatrix}
w_1\\w_2
\end{pmatrix}
&\quad& \text{in } \Omega \times (t_n,t_{n+1}), \\
w(t_n) &= U^{n**}. &&
\end{alignat*}
For each of the two components of $w=(w_1,w_2)$ and each $\jj$,
the characteristic  $z_{\jj}(\cdot)$ through $x_{\jj}$ solves the problem
\begin{align}
\label{characteristic}
\partial_t z_{\jj}(t) &= - \AA(z_{\jj}(t),t)
\quad \text{for } t \in (t_n,t_{n+1}), \\
z_{\jj}(t_{n+1}) &= x_{\jj}.
\end{align}
with end value prescribed at $t=t_{n+1}$.
Solving the above characteristic equation for each grid point $x_{\jj}$
would yield the sought approximation $U^{n***}(x_{\jj})$ via
\begin{equation*}
U^{n***}(x_{\jj})
= \mathrm{e}^{\timestepsize \Cfrak} U^{n**}(x_{\jj})
:= w(z_{\jj}(t_n),t_{n})
= U^{n**}(z_{\jj}(t_n)).
\end{equation*}
However, the point $z_{\jj}(t_n)$ is not a grid point in general,
so we do not have immediate access to the value $U^{n**}(z_{\jj}(t_n))$.
We need to use an interpolation method to approximate $U^{n**}(z_{\jj}(t_n))$
based on the knowledge of $U^{n**}$ at  grid points.
Since the previous step passes Fourier data to the advection step,
it is natural to use Fourier interpolation to accomplish this.
Following~\cite[Section~5.1]{caliari_splitting_2017}, we evaluate a Fourier interpolation at $x = z_{\jj}(t_n)$,
where the coefficients $\{ \widehat U^{n**}_{k_1,k_2,k_3} \}$ are known
from step~(ii) of Algorithm~\ref{algorithm}. In general, further choices are required to make such a trigonometric interpolation unique in a sensible way (see, e.g. \cite{johnson_2011}) but we omit discussion of this here - minimally oscillatory interpolations are usually to be preferred.
Besides this uniform trigonometric method, one could employ other methods for the interpolation,
e.g. the computationally more efficient non-uniform NUFFT-based approaches
as in \cite[Section~5.3]{caliari_splitting_2017} and~\cite[Section~2.2]{mzz2017}.

\subsection*{The coupling step}
The coupling step contains the off-diagonal components of the Pauli equation.
Step~(iv) of Algorithm~\ref{algorithm} consists in finding,
for all grid points $x_{\jj}$,
the solution
of the following initial value problem:
\begin{alignat*}{2}
\partial_s \begin{pmatrix}w_{\jj 1}(s) \\ w_{\jj 2}(s) \end{pmatrix}
& =
\begin{pmatrix}
0 & \Dfrak_1(s) \\
\Dfrak_2(s) & 0
\end{pmatrix}
\begin{pmatrix}w_{\jj 1}(s) \\ w_{\jj 2}(s) \end{pmatrix}
&\quad& \text{for } s \in (0,\timestepsize), \\
w_{\jj}(0) &= U^{n***}(x_{\jj}), &&
\end{alignat*}
where $w_{\jj}=(w_{\jj 1},w_{\jj 2})$ and
\begin{align*}
\Dfrak_1(s) & = \frac{i}{2} B_1(x_{\jj},t_n+s) + \frac{1}{2} B_2(x_{\jj},t_n+s),\\
\Dfrak_2(s) & = \frac{i}{2}B_1(x_{\jj},t_n+s) - \frac{1}{2} B_2(x_{\jj},t_n+s).
\end{align*}
Then, the solution of the coupling step,
which is also the approximation $U^{n+1} \approx u(t_{n+1})$,
is given by $U^{n+1}(x_{\jj}) = \mathrm{e}^{\timestepsize \Dfrak} U^{n***}(x_{\jj}) := w_{\jj}(\timestepsize)$.
Unlike the previous steps, this is a \emph{coupled} system of ODEs,
which may be treated with appropriate highly efficient solvers. An analytic solution to this ODE is readily available in each time step, and as with the potential step (step~(i) of Algorithm~\ref{algorithm}),
for the case of time-independent potentials the solution operator may in fact be pre-computed for all considered time-steps outside of the solution loop.

\end{document}